\documentclass[reqno]{amsart}
\usepackage[utf8]{inputenc}
\usepackage[T1]{fontenc}
\usepackage[english]{babel}
\usepackage{mathtools}
\usepackage{amssymb}
\usepackage{amsthm}
\usepackage{MnSymbol}
\usepackage{bbm}
\usepackage{amssymb,fge}
\mathtoolsset{showonlyrefs,showmanualtags}

\newtheorem{Satz}{Satz}[section]
\newtheorem{Lem}[Satz]{Lemma}

\newtheorem{Cor}[Satz]{Corollary}

\newtheorem{Thm}[Satz]{Theorem}

\newtheorem{Rem}[Satz]{Remark}

\newcommand{\R}{\mathbb{R}}

\newcommand{\inv}{^{-1}}
\newcommand{\Leb}{\mathcal{L}}

\newcommand{\K}{{S^1}}
\newcommand{\Dc}{\bar{D}^2}
\newcommand{\Do}{D^2}

\newcommand{\tn}[1]{\textnormal{#1}}

\newcommand{\tr}{\textrm{d}}
\makeatletter

\newcommand{\Rom}[1]{\expandafter\@slowromancap\romannumeral #1@}
\makeatother
\title{Plateau's problem for singular curves}
\author{Paul Creutz}
\address{Paul Creutz, Mathematisches Institut der Universit\"at zu K\"oln, Weyertal 86-90, 50931 K\"oln, Germany}
\email{pcreutz@math.uni-koeln.de}
\thanks{The author was partially supported by the DFG grant SPP 2026.}
\begin{document}
\begin{abstract}
We give a solution of Plateau's problem for singular curves possibly having self-intersections. The proof is based on the solution of Plateau's problem for Jordan curves in very general metric spaces by Alexander Lytchak and Stefan Wenger and hence works also in a quite general setting. However the main result of this paper seems to be new even in $\R^n$.
\end{abstract}
\maketitle
\section{Introduction}
\label{1}
\subsection{Main Result}
\label{1.1}
Classical formulation of Plateau's problem reads: Does there exist a disc of least area spanning a given closed curve $\Gamma$ in $\R^n$? If yes, what can be said about regularity of such discs?\par 
The first question has been positively answered by Douglas and Radó independently for general Jordan curves in 1930 and the second one and its variants have been extensively studied in the last 90 years, see \cite{DHS10} and \cite{DHT10} and the references therein for a detailled account on the literature. However the classical approaches break down if instead of Jordan curves one considers closed curves having self-intersections. Reason for this being that there need not be a disc of least energy spanning $\Gamma$. Joel Hass attacked the question of finding a least area disc for such singular curves in \cite{Has91} and showed by cutting into possibly infinitely many curves, that an area minimizing disc for $\Gamma$ exists in the sense of continuous maps and Lebesgue notion of area. However the resulting disc a priori does not satisfy any regularity at the possibly large preimage of $\Gamma$ beyond continuouity and does not quite fit into the classical setting as it might not be a Sobolev map. Using results for Plateau's problem in singular metric spaces given by Alexander Lytchak and Stefan Wenger in \cite{LW17a} we are able to obtain the following result. 
\begin{Thm}
\label{thm1}
Let $\Gamma$ be a closed rectifiable curve in $\R^n$ possibly having self-intersections. Then there exists $p>2$ and a Sobolev disc $u\in W^{1,p}(\Do,\R^n)$ spanning~$\Gamma$ of least area among all discs $v\in W^{1,2}(\Do,X)$ spanning $\Gamma$. One may choose $u$ locally $\alpha$-Hölder continuous on $\Do$ and globally $\beta$-Hölder continuous on $\Dc$ where $\alpha=\frac{1}{3}$ and $\beta=\frac{1}{27}$.
\end{Thm}
Here by a \textit{closed curve} $\Gamma$ we mean an equivalence class of parametrized closed curves $\gamma:\K \rightarrow \R^n$ where we identify two parametrized curves if they the same constant speed parametrization up to isometry of $\K$. We say that a Sobolev disc $u \in W^{1,2}(\Do,\R^n)$ \textit{spans} $\Gamma$ if $\tn{tr}(u)\in L^2(\K, \R^n)$ has a continuous representative which is an element of $\Gamma$. The area of $u \in W^{1,2}(\Do,\R^n)$ is given by
\[
\tn{Area}(u):=\int_{\Do} |u_x\wedge u_y|\ \tr x \ \tr y.
\]
Theorem~\ref{thm1} is not only interesting for self-intersecting curves but also seems to be new for rectifiable Jordan curves of low regularity. Let $\Gamma$ be a rectifiable Jordan curve in $\R^n$ and $u\in W^{1,2}(\Do,\R^n)$ be a disc of least area spanning $\Gamma$ classically obtained by minimizing the Dirichlet energy. Then $u$ satisfies Laplace equation and hence is smooth on $\Do$ which is much stronger than local Hölder continuouity, see \cite{DHS10}. If $\Gamma$ furthermore satisfies a $\lambda$-chord-arc condition, then $u$ is known to be globally $\beta$-Hölder continuous for some $\beta \in (0,1)$ depending on $\lambda$, see \cite{HvdM99}. If $n=3$ and $\Gamma$ is $C^1$ then one may take any $\beta<\frac{1}{2}$ by \cite[p.238]{Nit65}. Imposing higher regularity on $\Gamma$ one may gain also higher boundary regularity of $u$, see for example \cite{Nit69}, \cite{War70}. However for a general rectifiable Jordan curve $\Gamma$ in $\R^n$ the only boundary regularity of $u$ that seems to be known is $u\in C^0(\Dc,\R^n)$.\par 
The idea of proof is to replace $\R^n$ by the metric space $X$ obtained by attaching a collar to $\R^n$ along $\Gamma$. Now we use the results in \cite{LW17a} to solve Plateau's problem for a certain regular Jordan curve in $X$. A simple projection argument completes the proof. This way we trade in the singularity of the original curve $\Gamma$ for the singularity of $X$.
\subsection{Sketch of proof}
\label{1.2}
Giving the 'right' definitions, Plateau's problem can almost verbatim be asked not only in $\R^n$ but for a curve $\Gamma$ in a general complete metric space $X$, see section~\ref{2}. If $X$ is proper and $\Gamma$ a Jordan curve, Plateau's problem has been solved by Alexander Lytchak and Stefan Wenger in \cite{LW17a}. This result has been generalized to Jordan curves in many locally non-compact spaces by Chang-Yu Guo and Stefan Wenger in \cite{GWar}. To obtain regularity of the maps they additionally assume the space to satisfy a quadratic isoperimetric inequality. A metric space $X$ is said to satisfy a \textit{$C$-quadratic isoperimetric inequality} if for every parametrized Lipschitz curve $\gamma:\K \rightarrow X$ there exists $u \in W^{1,2}(\Do,X)$ such that $\tn{tr}(u)=\gamma$ and
\[
\tn{Area}(u)\leq C \cdot l(\gamma)^2.
\]
Instead of cutting the curve as Hass does in \cite{Has91} the idea of proving theorem~\ref{thm1} is to resolve the self-intersections by gluing a collar. Assume for simplicity $\Gamma$ is a closed curve of length $2\pi$ in $\R^n$. Let $\gamma:\K\rightarrow \R^n$ be a constant speed parametrization of $\Gamma$ and $X_\Gamma$ the metric space obtained by gluing a strip $\K \times [0,2\pi]$ to $\R^n$ along $(P,0)\sim \gamma(P)$. Then $X_\Gamma$ satisfies a quadratic isoperimetric inequality and the curve $\Gamma_{2\pi}$ corresponding to $\K \times \{2\pi\}$ is a Jordan curve in $X_\Gamma$ satisfying a chord-arc condition. So the results in \cite{LW17a} give a 'nice' disc $v$ spanning $\Gamma_{2\pi}$ of least area within the metric space $X_\Gamma$. Now there is a canonical $1$-Lipschitz retraction $P:X_\Gamma\rightarrow \R^n$. Then $P\circ v$ gives the desired disc spanning $\Gamma$. The details may be found in section~\ref{4.1}.\par
This proof is not very specific to $\R^n$. It gives the following more general result. 
\begin{Thm}
\label{thm2}
Let $X$ be a proper metric space satisfying a $C$-quadratic isoperimetric inequality and $\Gamma$ a closed rectifiable curve in $X$. Then there exists $p=p(C)>2$ and a disc $u\in W^{1,p}(\Do,X)$ of least area spanning $\Gamma$.
\end{Thm}
Theorem~\ref{thm2} may be generalized to the class of spaces which are \textit{$1$-complemented in some ultracompletion} considered in \cite{GWar}. It includes among others Hadamard spaces, dual Banach spaces, injective spaces and $L^1$-spaces. The proof is the same upon replacing the \cite{LW17a} results by the \cite{GWar} ones. We will only give the proof of the locally compact result and indicate where one has to make changes to obtain the more general version.\par
A similar construction of gluing a strip has also turned out to be useful in \cite{Crearb} and \cite{Sta18}.
\subsection{Parametrized version and applications}
\label{1.3}
Let $X$ be a proper metric space and $\Gamma$ a closed curve in $X$. Let 
\[
\tn{Fill}(\Gamma):=\min \{ \tn{Area}(u)| u \tn{ spans } \Gamma\}.
\]
Now let $\gamma:\K \rightarrow X$ be a parametrization of $\Gamma$. A subtle but apparent question is whether there exists $u$ such that $\tn{tr}(u)=\gamma$ and $\tn{Area}(u)=\tn{Fill}(\Gamma)$. In the \cite{Has91} setting the answer is always yes. However the procedure discussed there produces a map $u$ which might not be in $W^{1,2}(\Do,X)$. A positive answer to the question has been given in \cite{LW16} if $\Gamma$ is a Jordan curve satisfying a chord-arc condition and $\gamma$ a Lipschitz parametrization. We use this result and the same trick as in the proof of theorem~\ref{thm2} to obtain the following.
\begin{Thm}
\label{thm3}
Let $X$ be a proper metric space satisfying a $C$-quadratic isoperimetric inequality and $\Gamma$ a closed rectifiable curve in $X$. If $\gamma:\K \rightarrow X$ is a Lipschitz parametrization of $\Gamma$, then there exists a disc $u \in W^{1,p}(\Do,X)$ such that $\tn{tr}(u)=\gamma$ and $\tn{Area}(u)=\tn{Fill}(\Gamma)$ for some $p>2$ depending only on $C$.
\end{Thm}
Especially if a metric space satisfies a $C'$-quadratic isoperimetric for every $C'>C$, then it satisfies a $C$-quadratic isoperimetric inequality. This simplifies various things in the works of Alexander Lytchak and Stefan Wenger. For example one may slightly improve theorem~$1.8$ in \cite{LWYar} and theorem~1.4 in \cite{LW17b} as follows.
\begin{Cor}
\label{cor1}
Let $C>0$ and $(X_n)$ be a sequence of proper, geodesic metric spaces satisfying a $C$-quadratic isoperimetric inequality. If $X_\omega$ is an ultralimit of $(X_n)$, then $X_\omega$ satisfies a $C$-quadratic isoperimetric inequality. 
\end{Cor}
\begin{Cor}
\label{cor2}
Let $X$ be a complete, geodesic metric space that is homeomorphic to a $2$-dimensional manifold. Then $X$ satisfies a $C$-quadratic isoperimetric inequality iff for every Jordan curve $\Gamma$ in $X$ there exists a Jordan domain $U\subset X$ bounded by $\Gamma$ such that
\[
\mathcal{H}^2(U)\leq C\cdot l(\Gamma)^2.
\]
\end{Cor}
\section{Reminder: Sobolev maps}
\label{2}
In this section we give a short reminder on metric space valued Sobolev maps. For more details see for example \cite{LW17a}, \cite{KS93} and \cite{Res97}.\par 
Let $\Omega\subset \R^2$ a bounded domain, $X$ a separable complete metric space and $p>1$. A measurable, essentially separably valued map $u:\Omega \rightarrow X$ belongs to $W^{1,p}(\Omega,X)$ if for every $1$-Lipschitz map $g:X\rightarrow \R$ the composition $g\circ f$ belongs to the classical Sobolev space $W^{1,p}(\Omega,\R)$ and its \textit{Reshtnyak p-energy}
\begin{equation}
\label{eq1}
E^p_+(u):=\inf \left\{||g||^p_{L^p(\Omega)}\ \Big|\ g \in L^p(\Omega):\forall f\in \tn{Lip}_1(X,\R): |\nabla (f\circ u)|\leq g\tn{ a.e.}\right\}
\end{equation}
is finite. We say that $u:\Omega \rightarrow X$ belongs to $W^{1,p}_{\tn{loc}}(\Omega,X)$ if for every $\Omega'\subset \Omega$ precompact one has $u\in W^{1,p}(\Omega',X)$.\par 
If $\Omega$ is a Lipschitz domain, then for $u\in W^{1,p}(\Omega,X)$ there is a canonical almost everywhere defined \textit{trace} map $\tn{tr}(u)\in L^p(\partial \Omega, X)$. If $u$ extends continuously to a map $\bar{u}:\bar{\Omega}\rightarrow X$, then one may simply take $\tn{tr}(u)=\bar{u}_{|\partial \Omega}$.\par 
If $u\in W^{1,p}(\Omega,X)$, then $u$ is approximately metrically differentiable almost everywhere. That is for almost every $p \in \Omega$ there exists a seminorm $\tn{apmd}_p(u)$ on $\R^2$ such that
\begin{equation}
\label{eq2}
\tn{aplim}_{q\rightarrow p} \frac{d(u(p),u(q))-\tn{apmd}_p (q-p)}{|q-p|}=0.
\end{equation}
Define the \textit{(Busemann) area} of $u \in W^{1,p}(\Omega,X)$ by
\begin{equation}
\label{eq3}
\tn{Area}(u):=\int_\Omega J(\tn{apmd}_p(u))\ \tr \Leb^2(p).
\end{equation}
where for a seminorm $\sigma$ on $\R^2$ one sets its \textit{(Busemann) Jacobian} to be
\begin{equation}
\label{eq4}
J(\sigma):=\frac{\pi}{\Leb^2(\{v\in \R^2|\sigma(v)\leq 1\})}.
\end{equation}
This equals the usual definition $\tn{Area}(u):=\int_\Omega |u_x\wedge u_y|\ \tr x \tr y$ in case $X=\R^n$. We say that a map $u:\Omega\rightarrow X$ satisfies \textit{Lusin's property (N)} if for every $N\subset \Omega$ with $\Leb^2(N)=0$ one has $\mathcal{H}^2(u(N))=0$. If $u\in W^{1,p}(\Omega,X)$ satisfies Lusin's property (N), then the following variant of the area formula holds
\begin{equation}
\label{eq5}
\tn{Area}(u)=\int_X \tn{card}( u\inv (x)) \ \tr \mathcal{H}^2(x).
\end{equation}
A map $u\in W^{1,2}(\Omega,X)$ is \textit{$Q$-quasi-conformal} if for $\Leb^2$-almost every $p \in \Omega$ one has 
\begin{equation}
(\tn{ap md}_p u )(v)\leq Q \cdot (\tn{ap md}_p u )(w)
\end{equation}
for all $v,w \in \K$. If $u$ is $Q$-quasi-conformal, then one has
\begin{equation} 
\tn{Area}(u)\leq E^2_+(u)\leq Q^2\cdot \tn{Area}(u).
\end{equation}
Also the following variant of the Sobolev embedding theorem holds.
\begin{Thm}
\label{thm4}
Let $p>2$, $\alpha=1-\frac{2}{p}$ and $u\in W^{1,p}_{\tn{loc}}(\Omega,X)$. Then $u$ has a representative $\bar{u}\in C^\alpha_{\tn{loc}}(\Omega,X)$ that satisfies Lusin's property (N). If furthermore $\Omega$ is a Lipschitz domain, then
\begin{equation}
||\bar{u}||_{\alpha}=\sup_{x,y \in \Omega} \frac{d(\bar{u}(x),\bar{u}(y))}{|x-y|^\alpha}\leq K(p, \Omega)\cdot \left(E^p_+(u)\right)^{\frac{1}{p}}.
\end{equation}
So especially $W^{1,p}(\Omega,X)\subset  C^\alpha(\bar{\Omega},X)$.
\end{Thm}
\section{Plateau's problem for Jordan curves}
\label{3}
\subsection{Unparametrized Plateau problem}
\label{3.1}
Consider $\K$ endowed with angular distance. For $p,q \in \K$ let $[p,q]$ be a shortest path in $\K$ connecting $p$ and $q$. We say that a parametrized curve $\gamma:\K \rightarrow X$ is of \textit{constant speed} if there is a constant $c\in [0,\infty)$ such that for every $p,q \in \K$ one has $l(\gamma_{|[p,q]})=c\cdot d_\K(p,q)$. For every rectifiable parametrized curve $\gamma:\K \rightarrow X$ there exists a constant speed curve $\bar{\gamma}:\K\rightarrow X$ and a monotone map $R:\K \rightarrow \K$ such that $\gamma=\bar{\gamma}\circ R$. Such a curve $\bar{\gamma}$ is unique up to precomposing an isometry of $\K$ and will be called the \textit{constant speed parametrization} of $\gamma$.  Having the same constant speed parametrization defines an equivalence relation~$\sim$ on the set of parametrized rectifiable curves $\gamma:\K\rightarrow X$. A \textit{closed rectifiable curve} $\Gamma$ is an equivalence class with respect to $\sim$. We say that $\gamma:\K \rightarrow X$ is a \textit{parametrization} of $\Gamma$ if $\gamma \in \Gamma$.\par
Let $\Gamma$ be a closed rectifiable curve. We say that $u \in W^{1,2}(\Do,X)$ \textit{spans} $\Gamma$ if $\tn{tr}(u)\in L^2(\Do,X)$ has a continuous representative that is a parametrization of $\Gamma$. We define the \textit{filling area} $\tn{Fill}(\Gamma)$ of $\Gamma$ by \[
\tn{Fill}(\Gamma):=\inf\{\tn{Area}(u)|u \tn{ spans } \Gamma\}.
\]
In this formulation \textit{Plateau's problem} for $X$ reads: Is there a disc $u$ spanning $\Gamma$ such that $\tn{Area}(u)=\tn{Fill}(\Gamma)$?\par 
The classical approach to solve Plateau problem for Jordan curves $\Gamma$ in $\R^n$ is to minimize the energy over all discs spanning $\Gamma$. Then it turns out that a minimizer of energy exists and is also a minimizer of area. This breaks down for more general metric spaces as energy minimizers might not be area minimizers and worse for self-intersecting curves as there might not even exist energy minimizers, see \cite{LW17a} resp. \cite{Has91}. Still Alexander Lytchak and Stefan Wenger were able to solve Plateau's problem for Jordan curves in proper metric spaces.
\begin{Thm}[\cite{LW17a}]
\label{thm5}
Let $X$ be a proper metric space and $\Gamma$ a rectifiable Jordan curve in $X$ such that $\tn{Fill}(\Gamma)<\infty$. Then there exists a disc $u\in W^{1,2}(\Do,X)$ of least area spanning $\Gamma$. One may furthermore choose $u$ to minimize $E^2_+(u)$ over all discs of least area spanning $\Gamma$. In this case $u$ is $\sqrt{2}$-quasi-conformal. If $X$ satisfies property~(ET), $u$ is even $1$-quasi-conformal.
\end{Thm}
Here a metric space is called \textit{proper} if closed and bounded subsets of $X$ are compact.
A metric space is said to satisfy \textit{property (ET)} if for every $u \in W^{1,2}(\Do,X)$ the metric differential $\tn{apmd} f$ is induced by a (possibly degenerate) inner product $\Leb^2$-almost everywhere on $\Do$. Spaces satisfying property (ET) include complete Riemannian manifolds and spaces satisfying curvature bounds in the sense of Alexandrov, see \cite{LW17a}.\par 
Theorem~\ref{thm5} has been generalized to many locally non-compact spaces including Hadamard spaces, dual Banach spaces and $L^1$-spaces in \cite{GWar}. It holds for all complete metric spaces $X$ that are \textit{$1$-complemented in some ultracompletion}. For convenience of the reader we stick in this article to proper spaces. However all results stated here for proper spaces also hold true within the class of spaces $1$-complemented in some ultra completion.\par 
To obtain regularity of solutions of Plateau's problem one has to impose additional conditions on $X$. Let $C\in [0,\infty)$ and $l_0\in (0,\infty]$. We say that $X$ satisfies a \textit{$(C,l_0)$-quadratic isoperimetric inequality} if for every parametrized Lipschitz curve $\gamma:\K \rightarrow X$ such that $l(\gamma)<l_0$ there exists a disc $u \in W^{1,2}(\Do,X)$ such that $\tn{tr}(u)=\gamma$ and
\[
\tn{Area}(u)\leq C \cdot l(\gamma)^2.
\]
If $l_0=\infty$ we also say that $X$ satisfies a $C$-quadratic isoperimetric inequality. For example $\R^n$ and more generally $\tn{CAT}(0)$ spaces satisfy a $\frac{1}{4\pi}$-quadratic isoperimetric inequality. For proper, geodesic spaces this is an equivalent characterization of the $\tn{CAT}(0)$-condition, see \cite{LW18b}. Banach spaces satisfy a $\frac{1}{2\pi}$-quadratic isoperimetric inequality, see \cite{Creara}.
\begin{Thm}[\cite{LW17a}, \cite{LW16}]
\label{thm6}
Let $X$ be a complete metric space satisfying a $(C,l_0)$-quadratic isoperimetric inequality and $\Gamma$ a rectifiable Jordan curve in $X$. Let $u\in W^{1,2}(\Do,X)$ be a disc of least area spanning $\Gamma$ that is $Q$-quasiconformal. Then
\begin{enumerate}
\item Then $u\in W^{1,p}_{\tn{loc}}(\Do,X)\cap C^\alpha_{\tn{loc}}(\Do,X)$ for $p=p(C,Q)>2$ and $\alpha=(4\pi Q^2 C)\inv$.
\item If $\Gamma$ satisfies a $\lambda$-chord-arc condition, then $u\in W^{1,q}(\Dc,X)\cap C^\beta(\Dc,X)$ where $q=q(C,Q,\lambda)>2$ and $\beta=\frac{\alpha}{(1+2\lambda)^{2}}$. 
\end{enumerate} 
\end{Thm}
Here a rectifiable Jordan curve $\Gamma$ is said to satisfy a $\lambda$-\textit{chord-arc} condition where $\lambda \geq 1$ is a constant if its constant speed parametrization $\bar{\gamma}:\K \rightarrow X$ satisfies
\begin{equation}
\label{eq6}
\frac{l(\Gamma)}{2\pi}\cdot d_\K(p,q)\leq \lambda \cdot d(\bar{\gamma}(p),\bar{\gamma}(q))
\end{equation}
for all $p,q \in \K$.
\subsection{Parametrized Plateau's problem}
\label{3.2}
Let $\Gamma$ be a rectifiable curve in $X$ and $\gamma:\K \rightarrow X$ a parametrization of $\Gamma$. We define the \textit{filling area} of $\gamma$ to be
\begin{equation}
\label{eq7}
\tn{Fill}(\gamma):=\inf\{\tn{Area}(u)\ |\ \tn{tr}(u)=\gamma\}.
\end{equation}
Then by definition one has $\tn{Fill}(\Gamma)\leq \tn{Fill}(\gamma)$. A first apparent question is whether equality holds. This cannot be true because even in $\R^n$ as a rectifiable curve will usually have bad parametrizations that do not admit discs spanning them. However this is essentially the only restriction by the following result of Lytchak and Wenger.
\begin{Thm}[\cite{LW18a}, lemma~$4.8$]
\label{thm7}
Let $X$ be a complete metric space satisfying a $(C,l_0)$-quadratic isoperimetric inequality, $\Gamma$ a closed rectifiable curve in $X$ and $\gamma$ a parametrization of $\Gamma$. If $\tn{Fill}(\gamma)<\infty$, then $\tn{Fill}(\gamma)=\tn{Fill}(\bar{\gamma})=\tn{Fill}(\Gamma)$.
\end{Thm}
However the proof of theorem~\ref{thm7} in \cite{LW18a} is quite tricky and does not give the existence of a disc $u$ such that $\tn{tr}(u)=\gamma$ and $\tn{Area}(u)=\tn{Fill}(\gamma)$. The question whether such $u$ exists could be considered a parametrized version of Plateau's problem. The following partial result has been obtained in \cite{LW16}.
\begin{Thm}[\cite{LW16}]
\label{thm8}
Let $X$ be a proper metric space satisfying a $(C,l_0)$-quadratic isoperimetric inequality, $\Gamma$ a rectifiable $\lambda$-chord-arc curve in $X$ and $\gamma$ a Lipschitz parametrization of $\Gamma$. If $\tn{Fill}(\Gamma)<\infty$, then there exists $p=p(C,\lambda)>2$ and a disc $u\in W^{1,p}(\Do,X)$ such that $\tn{tr}(u)=\gamma$ and $\tn{Area}(u)=\tn{Fill}(\Gamma)$. If furthermore $l(\gamma)<l_0$, then one may also achieve
\begin{equation}
\left(E^p_+(u)\right)^{\frac{1}{p}}\leq M \cdot \tn{Lip}(\gamma)  
\end{equation}
where $M=M(C,\lambda)$.
\end{Thm}
The following simple gluing lemma is crucial in the proofs of theorem~\ref{thm7} and theorem~\ref{thm8} and will also play an important role in the proof of theorem~\ref{thm2}.
\begin{Lem}
\label{lem1}
Let $X$ be a complete metric space, $\gamma_1,\gamma_2:\K \rightarrow X$ be parametrized curves and $p>1$. Let $h\in C^0(\K \times [0,1],X)\cap W^{1,p}(\K\times (0,1),X)$ be a homotopy between $\gamma_1$ and $\gamma_2$ and let $u \in W^{1,p}(\Do,X)$ such that $\tn{tr}(u)=\gamma_1$. Then there exists a disc $v\in W^{1,p}(\Do,X)$ such that $\tn{tr}(u)=\gamma_2$, \[
\tn{Area}(v)=\tn{Area}(u)+\tn{Area}(h)\]
and
\[E^p_+(v)\leq L\cdot (E^p_+(u)+E^p_+(h))\]
 where $L=L(p)$.
\end{Lem}
See for example section 2.2 in \cite{LW16}.
\section{Plateau's problem for singular curves}
\label{4}
\subsection{Unparametrized variant}
\label{4.1}
Let $X$ be a complete metric space and $\Gamma$ a rectifiable curve in $X$. Let $l:=l(\Gamma)$, $S$ a geodesic circle of circumference $l$ and $\gamma:\K\rightarrow X$ be a constant speed parametrization of $\Gamma$. Consider $\gamma$ a $1$-Lipschitz map $S\rightarrow X$. Let $X_{\Gamma}$ be the metric quotient of $X\sqcup (S\times[0,l])$ along the relation generated by $\gamma(p)\sim (p,0)$ for $p\in S$. Let $P_\Gamma:X_\Gamma\rightarrow X$ be given by the identity on $X$ and $P_\Gamma(p,t):=\gamma(p)$ on $S\times [0,l]$.
\begin{Lem}
\label{lem2}
Let $X_\Gamma$ as described.
\begin{enumerate}
\item \label{en1}
$P_\Gamma:X_\Gamma\rightarrow X$ is $1$-Lipschitz, the inclusion map $\iota_X:X\rightarrow X_{\Gamma}$ is an isometric embedding and the inclusion map $\iota_{S\times[0,l]}:S\times [0,l]\rightarrow X_\Gamma$ is $1$-Lipschitz and locally isometric on $S\times (0,l]$.
\item \label{en2}
Let $\Gamma_l$ be the curve in $X_{\Gamma}$ corresponding to $S\times \{l\}$ in $X_\Gamma$. Then $\Gamma_l$ is a $1$-chord-arc curve of length $l$.
\item \label{en3}
If $X$ is proper, then $X_\Gamma$ is proper.
\item \label{en4}
If $X$ satisfies $(C,l_0)$-quadratic isoperimetric inequality, then $X_{\Gamma}$ satisfies a $(C+1,l_0)$-quadratic isoperimetric inequality and a $\left(C+\frac{1}{2\pi},l_0'\right)$-quadratic isoperimetric inequality where $l_0'=\min\{l_0,l\}$.
\item \label{en5}
If $X$ satisfies property (ET), then $X_{\Gamma}$ satisfies property (ET).
\end{enumerate}
\end{Lem}
\begin{Rem}
\label{rem1}
For the locally non-compact version of the theorems one has to note: if $X$ is $1$-completed in a ultra completion $X^\omega$, then $X_\Gamma$ is $1$-complemented in $X_\Gamma^\omega$. This is an easy consequence of equation \eqref{eq8} and the fact that $S\times [0,l]$ is compact.
\end{Rem}
\begin{proof}
Let $Y:=S\times [0,l]$ and $d_X$ and $d_Y$ be the distance functions of $X$ and $Y$ respectively. The distance $d$ on $X_\Gamma$ is given as follows, see \cite{Crearb}.
\begin{equation}
\label{eq8}
d([x],[y])=\begin{cases}
\scriptstyle \underset{p,q\in S}{\min}\left\{d_{Y}(x,y),\ d_Y(x,(p,0))+d_X(\gamma(p),\gamma(q))+d_Y((q,0),y)\right\} &; x,y\in Y\\
\underset{p\in S}{\min}\ d_X(x,\gamma(p)) +d_Y((p,0),y) &; x \in X, y \in Y\\
d_X(x,y) &; x,y \in X
.\end{cases}
\end{equation}
\ref{en1}. and \ref{en2}. are direct consequences of \eqref{eq8}. \ref{en3}. is a consequence of \ref{en1}. For \ref{en4}. see \cite{Crearb}. To see \ref{en5}. note that $X$ has property (ET) by assumption and $X_\Gamma\setminus X$ has property (ET) as it is locally isometric to $\R^2$.
\end{proof}
We prove the following stronger variant of theorem~\ref{thm2}.
\begin{Thm}
\label{thm9}
Let $X$ be a complete metric space satisfying a $(C,l_0)$-quadratic isoperimetric inequality and $\Gamma$ a closed rectifiable curve in $X$ such that $\tn{Fill}(\Gamma)<\infty$. Then there is a disc $u$ of least area spanning $\Gamma$. One may choose \[
u\in W^{1,p}(\Do,X_\Gamma)\cap C^\alpha_\tn{loc}(\Do,X_\Gamma)\cap C^\beta(\Dc,X_\Gamma)
\] 
where $p=p(C)>2$, $\alpha=\left(8\pi C+4\right)\inv$ and $\beta=\frac{\alpha}{9}$. If $X$ satisfies property (ET), then $\alpha$ may be improved to $\alpha=(4\pi C+2)\inv$.
\end{Thm}
\begin{proof}
Let $\gamma_l$ be the constant speed parametrization of $\Gamma_l$. Then $\iota_{S\times[0,l]}:S\times [0,l]\rightarrow X_\Gamma$ gives a homotopy between $\gamma$ and $\gamma_l$ of area $l^2$. So by theorem~\ref{thm7} and lemma~\ref{lem1}
\begin{equation}
\label{eq9}
\tn{Fill}(\Gamma_l)=\tn{Fill}(\gamma_l)\leq\tn{Fill}(\gamma)+l^2= \tn{Fill}(\Gamma)+l^2.
\end{equation}
Note here that as $X$ is a $1$-Lipschitz retract of $X_\Gamma$ the quantities $\tn{Fill}(\gamma)$ and $\tn{Fill}(\Gamma)$ are the same when considered within $X$ and within $X_\Gamma$.\par 
By theorem~\ref{thm5} there exists a disc of least area $v\in W^{1,2}(\Do,X_\Gamma)$ spanning $\Gamma_l$ that is $Q$-quasiconformal, where $Q=1$ if $X$ satisfies property (ET) and $Q=\sqrt{2}$ otherwise. By lemma~\ref{lem2}, theorem~\ref{thm6} and theorem~\ref{thm4} we may furthermore assume that
\[
v\in W^{1,p}(\Do,X_\Gamma)\cap C^\alpha_{\tn{loc}}(\Do,X_\Gamma)\cap C^\beta(\Dc,X_\Gamma)
\] and $v$ satisfies Lusin's property (N) where $p=p(C)>2$, 
\[
\alpha=\left(4\pi Q^2 \left(C+\frac{1}{2\pi}\right)\right)\inv=(4\pi Q^2 C+2Q^2)\inv
\]
and $\beta=\frac{\alpha}{9}$. By continuity of $v$ on $\Dc$ one has $S\times [0,l] \subset u(\Dc)$. Then $u:=P_\Gamma\circ v$ is a disc spanning $\Gamma$ having the desired regularity properties. As $v$ and $u$ satisfy Lusin's property (N) by \eqref{eq5} one has
\begin{align}
\tn{Area}(u)&=\int_{X} \tn{card}( u\inv(y))\ \tr \mathcal{H}^2(y)=\int_{X} \tn{card}( v\inv(y))\ \tr \mathcal{H}^2(y)\\
\label{eq10}
&\leq \int_{X_\Gamma} \tn{card}( v\inv(y))\ \tr \mathcal{H}^2(y)-l^2=\tn{Area}(v)-l^2.
\end{align}
So by \eqref{eq9} and \eqref{eq10}
\begin{equation}
\label{eq11}
\tn{Fill}(\Gamma_l)\leq \tn{Fill}(\Gamma)+l^2 \leq \tn{Area}(u)+l^2 \leq \tn{Area}(v)=\tn{Fill}(\Gamma_l).
\end{equation}
So $\tn{Area}(u)=\tn{Fill}(\Gamma)$ which proves theorem~\ref{thm9}.
\end{proof}
Theorem~\ref{thm1} follows from theorem~\ref{thm9} by noting that $\R^n$ has property (ET) and satisfies a $\frac{1}{4\pi}$-quadratic isoperimetric inequality.
\begin{Rem}
\label{rem2}
Let $X$ be a $\tn{CAT}(0)$ space and $\Gamma$ a rectifiable curve in $X$ of finite total geodesic curvature $\kappa(\Gamma)$. In this case one may improve the constant $\alpha$ in theorem~\ref{thm9} to $1$ and hence the map $u$ will be locally Lipschitz on $\Do$. However the constant $\beta$ will depend on $\kappa(\Gamma)$ in this case. To achieve this one performs the same proof using the funnel extension discussed in section~3.1 of \cite{Sta18} instead of $X_\Gamma$.
\end{Rem}
\subsection{Parametrized variant}
\label{4.2}
The proof of the following variant of theorem~\ref{thm3} is very similar to the proof discussed in the previous subsection.
\begin{Thm}
\label{thm10}
Let $X$ be a proper metric space satisfying a $(C,l_0)$-quadratic isoperimetric inequality, $\Gamma$ a closed rectifiable curve in $X$ such that $\tn{Fill}(\Gamma)<\infty$ and $\eta$ be a Lipschitz parametrization of $\Gamma$.\par
Then there there exists $p=p(C)>2$ and a disc $u\in W^{1,p}(\Do,X)$ such that $\tn{tr}(u)=\gamma$ and $\tn{Area}(u)=\tn{Fill}(\Gamma)$. If $l(\Gamma)<l_0$, then one may furthermore achieve
\begin{equation}
\label{eq12}
\left(E^p_+(u)\right)^{\frac{1}{p}}\leq M \cdot \tn{Lip}(\gamma)
\end{equation}
where $M=M(C)$.
\end{Thm}
\begin{proof}
First assume $\eta=\gamma$ is the constant speed parametrization of $\Gamma$. Consider again the space $X_\Gamma$ as in the previous subsection. By theorem~\ref{thm8} there exists $p=p(C)>2$ and a disc $v\in W^{1,p}(\Do,X_\Gamma)$ such that $\tn{tr}(v)=\gamma_l$ and $\tn{Area}(v)=\tn{Fill}(\gamma_l)$. If $l(\Gamma)=l(\Gamma_L)<l_0$, then one may furthermore achieve
\begin{equation}
\label{eq13}
\left(E^p_+(u)\right)^\frac{1}{p}\leq M\cdot \tn{Lip}(\gamma_L)\leq  M\cdot\tn{Lip}(\eta)
\end{equation}
where $M=M(C)$. Set $u:=P_\Gamma \circ v$. Then $u\in W^{1,p}(\Do,X)\cap C^0(\Dc,C)$. Hence $\tn{tr}(u)=u_{|\K}=\eta$ and the same argument as in the proof of theorem~\ref{thm9} implies $\tn{Area}(u)=\tn{Fill}(\Gamma)$.\par 
Now for a general Lipschitz curve $\eta$ by \cite[lemma 3.6]{LWYar} there exists a universal constant $K>0$ and a Lipschitz homotopy $h:\K \times [0,1]\rightarrow X$ between $\eta$ and $\gamma$ such that $\tn{Area}(h)=0$ and $\tn{Lip}(h)\leq M \cdot \tn{Lip}(\eta)$. So the proof is completed by applying lemma~\ref{lem1} to the disc constructed for $\gamma$ and the homotopy $h$ . 
\end{proof}
As a consequence of theorem~\ref{thm10} if a proper space $X$ satisfies a $(C',l_0)$-quadratic isoperimetric for every $C'>C$ then it satisfies a $C$-quadratic isoperimetric inequality. Corollary~\ref{cor2} is an immediate consequence of this fact and theorem~$1.4$ in \cite{LW17b}.
\begin{Cor}
\label{cor3}
Let $(X_n)$ be a sequence of proper, geodesic metric spaces satisfying a $(C,l_0)$-quadratic isoperimetric inequality. If $X_\omega$ is an ultralimit of $(X_n)$, then $X_\omega$ satisfies a $(C,l_0)$-quadratic isoperimetric inequality. 
\end{Cor}
\begin{proof}
By \cite[theorem 5.1]{LWYar} $X_\omega$ satisfies a $C'$-quadratic isoperimetric inequality for every $C'>C$. So if $X_\omega$ is proper we are done. If not one may perform the same proof as in \cite[theorem 5.1]{LWYar} but applying our theorem~\ref{thm10} instead of theorem~5.2 therein.
\end{proof}
\section{Acknowledgements}
\label{5}
I would like to thank my PhD advisor Alexander Lytchak for great support in everything.
\bibliographystyle{alpha}
\bibliography{SingularPlateauProblem}
\end{document}